\documentclass[10pt,a4paper]{article}
\usepackage[utf8]{inputenc}

\usepackage[usenames,dvipsnames]{xcolor}
\usepackage[margin=1.0 in]{geometry}
\usepackage{graphicx}
\usepackage{amsmath, amsthm, amssymb}
\usepackage{appendix}
\usepackage{hyperref}
\usepackage{marginnote, enumitem}
\hypersetup{
    unicode=false,          
    pdftoolbar=true,        
    pdfmenubar=true,        
    pdffitwindow=false,     
    pdfstartview={FitH},    
    pdftitle={My title},    
    pdfauthor={Author},     
    pdfsubject={Subject},   
    pdfcreator={Creator},   
    pdfproducer={Producer}, 
    pdfkeywords={keyword1} {key2} {key3}, 
    pdfnewwindow=true,      
    colorlinks=true,       
    linkcolor=blue,          
    citecolor=blue,        
    filecolor=magenta,      
    urlcolor=cyan           
}

\usepackage[square, sort, numbers]{natbib}

\numberwithin{equation}{section}

\newtheorem{lem}{Lemma}[section]
\newtheorem{theo}{Theorem}[section]

\newtheorem{rmk}{Remark}[section]
\newtheorem{dfn}{Definition}[section]
\newtheorem{asmp}{Assumption}[section]

\newtheorem{ex}{Example}[section]

\newtheorem*{theo*}{Theorem}

\usepackage{amsmath, bbm, accents}
\usepackage{amsfonts}
\newcommand{\bbR}{\mathbb{R}}

\newcommand{\bbE}{\mathbb{E}}

\newcommand{\bbN}{\mathbb{N}}

\newcommand{\scrP}{\mathcal{P}}
\newcommand{\scrM}{\mathcal{M}}

\newcommand{\calB}{\mathcal{B}}

\newcommand{\calS}{\mathcal{S}}
\newcommand{\calF}{\mathcal{F}}
\newcommand{\calL}{\mathcal{L}}

\newcommand{\bbP}{\mathbb{P}}
\newcommand{\bbQ}{\mathbb{Q}}
\newcommand{\E}[2][]{\mathbb{E}_{#1}\left[#2\right]}

\newcommand{\crochet}[1]{\left\langle{#1}\right\rangle}
\newcommand{\accro}[1]{\left\{{#1}\right\}}

\newcommand{\abs}[1]{\left| {#1} \right|}

\newcommand{\wt}[1]{\widetilde{#1}}
\newcommand{\ind}[1]{\mathbbm{1}_{\accro{#1}}}

\usepackage{xr}
\title{Asymptotic Optimal Tracking: Feedback Strategies}
\date{}
\author{Jiatu Cai$^{1}$, Mathieu Rosenbaum$^{2}$ and Peter Tankov$^{1}$
\\[0.4cm]
{\normalsize
$^1$ Laboratoire de Probabilit\'es et Mod\`eles Al\'eatoires,} \\ {\normalsize
    Universit\'e Paris Diderot (Paris 7)} \\
$~~$\\
{\normalsize $^2$ Laboratoire de Probabilit\'es et Mod\`eles Al\'eatoires,} \\ {\normalsize
    Universit\'e Pierre et Marie Curie (Paris 6)}
}

\begin{document}
\maketitle
\begin{abstract}This is a companion paper to
  \cite{cai2015asymptotic}. 
We consider a class of strategies of feedback form for the problem of
tracking and study their performance under the asymptotic framework of
the above reference. The strategies depend only on the current state
of the system and keep the deviation from the target inside a time-varying domain. Although the dynamics of the target is non-Markovian, it turns out that such strategies are asympototically optimal for a large list of examples.  \\\\
\noindent \textbf{Key words:}\ {Tracking problem, asymptotic optimality, feedback strategies}
\end{abstract}


\section{Introduction}

\noindent We consider the problem of tracking a target whose dynamics $(X^\circ_t)$ is modeled by a continuous It\=o semi-martingale defined on a filtered probability space $(\Omega, \calF, (\calF_t)_{t\geq 0}, \bbP)$  with values in $\bbR^d$ such that
\begin{equation*}
dX^\circ_t =  b_t dt+ \sqrt{ a}_t dW_t,\quad X^\circ_0= 0.
\end{equation*}
Here, $(W_t)$ is a $d$-dimensional Brownian motion and $(b_t)$,
$(a_t)$ are  predictable processes with values in $\bbR^d$ and
$\calS^d_+$, the set of $d\times d$ symmetric positive definite
matrices, respectively. An agent observes $X^\circ_t$  and applies a
control $\psi_t$ to adjust her position $Y^\psi_t$ in order to follow
$X^\circ_t$.  The deviation from the target is denoted by  $X_t = -
X^\circ_t + Y^\psi_t$. The agent has to pay certain intervention costs for
position adjustments and aims to minimize the functional 
\begin{align}
J(\psi) = H_0(X) + H(\psi), \label{minfunc}
\end{align}
where $H(\psi)$ is the cost functional depending on the type
of control and $H_0$ is the deviation penalty
\begin{equation*}
H_0(X) = \int_0^T r_tD( X_t)dt,
\end{equation*}
where $(r_t)$ is a random weight process and $D(x)$ a deterministic
function. In the small-cost asymptotic of \cite{cai2015asymptotic}
(recalled below), minimization of the functional \eqref{minfunc} is
possible in the pathwise sense. \\

\noindent The agent has at her disposal several types of control with
associated cost functionals. 
\begin{itemize}
\item With \emph{impulse control} the position of the agent is given
  by \begin{equation*}
Y_t = \sum_{0< \tau_j \leq t} \xi_j
\end{equation*}
and the cost functional is
\begin{equation*}
H(\psi) = \sum_{0< \tau_j \leq T} k_{\tau_j}F( \xi_j).
\end{equation*}
Here, $\{\tau_j,j\in \bbN^*\}$ is a strictly increasing sequence of
stopping times representing the jump times and satisfying $\lim_{j\to
  \infty} \tau_j = +\infty$, for each $j\in \bbN^*$, $\xi_j\in
\bbR^d$ is a $\calF_{\tau_j}$-measurable random vector representing
the size of $j$-th jump, $(k_t)$ is a random weight process and $F(\xi)>0$ is the cost of a jump.
\item With \emph{singular control}, the position of the agent is given
  by 
\begin{equation*}
Y_t = \int_0^t \gamma_s d\varphi_s,
\end{equation*}
and the corresponding cost is
\begin{equation*}
H(\psi) = \int_0^T h_t P(\gamma_t) d\varphi_t.
\end{equation*}
Here, $\varphi$ is a progressively measurable increasing process with
$\varphi_{0-}=0$, which represents the cumulated amount of
intervention, $\gamma$ is a progressively measurable process with
$\gamma_t \in \Delta:=\{n\in \bbR^d |\sum_{i=1}^d |n^i| = 1\}$ for all
$t\geq 0$, which represents the distribution of the control effort in
each direction, $(h_t)$ is a random  weight process and $P(\cdot)$
determines the cost of applying the control in a given direction.  
\item With \emph{regular control} the position of the agent is given by
\begin{equation*}
Y_t = \int_0^t u_s ds,
\end{equation*}
and the cost functional is
\begin{equation*}
H(\psi) = \int_0^T l_t Q( u_t) dt.
\end{equation*}
Here, $u$ is a progressively measurable integrable process
representing the speed of the agent, $(l_t)$ is a random weight
process and $Q(\cdot)$ is the running cost function. 
\end{itemize}

\noindent{Throughout this paper we assume that the cost functions $D$, $Q$, $F$, $P$ verify the following homogeneity property}
\begin{equation*}
D(\varepsilon x) = \varepsilon^{\zeta_D}D(x), \quad Q(\varepsilon u) = \varepsilon^{\zeta_Q}Q(u), \quad F(\varepsilon \xi) = \varepsilon^{\zeta_F}F(\xi), \quad P(\varepsilon \xi) = \varepsilon^{\zeta_P} P(\xi),
\end{equation*}
with $\zeta_D >0$, $\zeta_Q > 1$, $\zeta_F=0$, $\zeta_P = 1$. For example, we could take
\begin{equation*}
D(x) = x^T \Sigma^D x, \quad Q(u) = u^T \Sigma^Q u, \quad F(\xi) = \sum_{i=1}^d F_i \ind{\xi^i \neq 0}, \quad P(\xi) = \sum_{i=1}^d P_i |\xi^i|,
\end{equation*}
with $\min_i F_i >0$ and $\Sigma^D, \Sigma^Q\in \calS^d_+$ such that
$\zeta_D=\zeta_Q = 2$. 
We refer to \cite{cai2015asymptotic} for a more detailed description
of the setting as well as for a review of literature on the tracking
problem and on its applications to mathematical finance.  \\

\noindent Finding optimal control
policies for such systems, which have general non-Markovian dynamics
and complex cost structures, is in general infeasible and may not even
make sense in the pathwise setting. However, in
\cite{cai2015asymptotic} we were able to establish a lower bound for
the best achievable asymptotic performance under
a suitable asymptotic framework. We have shown that the lower bound is related to the time-average
control of Brownian motion.  It is then natural to try to construct
tracking policies that are (near-)optimal by suitably adapting the
solution of time-average control problem of Brownian motion. We will
show that this is indeed possible when the latter is
available. However, closed-form solutions of the time-average control
of Brownian motion are rarely available, and computing numerical
solutions may be costly. From a practical viewpoint, it is also often irrelevant to find \emph{the} optimal strategy since there are already many approximations in the model. \\

\noindent The aim of this paper is therefore to introduce and study feedback
strategies which are easy to implement in practice and whose
asymptotic performance is close to the lower bound of
\cite{cai2015asymptotic}. Here, ``feedback''
means that the control decision is Markovian, depending only on the
current state of the system. 
\\

\noindent We consider three control settings (combined regular and impulse
control; combined regular and singular control; only regular control),
and for each setting introduce a class of feedback strategies, and a
corresponding small cost asymptotic framework, which allows to
establish the convergence of the renormalized cost functional to a
well-defined limit under suitable assumptions. Comparing this limit with the lower bound of
\cite{cai2015asymptotic}, we can prove the asymptotic optimality of
the feedback strategy in a variety of examples, and quantify the
``asymptotic suboptimality'' in the situations when the
asymptotic optimality does not hold. \\

\noindent \emph{Notation.} 
 We denote the graph of an application $M: E \to E'$ by $M^g$. Let $(E, d)$ be a complete separable metric space. $\scrM(E)$ denotes the space of Borel measures on $E$ equipped with weak topology and $\scrP(E)$ denotes the subspace of probability measures. The Hausdorff distance on the space of closed sets is given by
\begin{equation*}
H(A, B) = \inf \{\delta, B\subseteq V_\delta(A)\text{ and }A\subseteq V_\delta(B)\},
\end{equation*}
where $V_\delta(\cdot)$ denotes the $\delta$-neighborhood of a closed
set, i.e. $V_\delta(A)= \{x\in E, d(x, A)\leq  \delta\}$. We shall
also sometimes need the standard notation from classical EDP theory: for a domain
$\Omega$, we denote by $C^{l+\alpha}(\Omega)$ the class of functions
which are bounded and $\alpha$-H\"older continuous on $\Omega$
together with their first $l$ derivatives, and we say that the domain
is of class $C^{l+\alpha}$ if its boundary may be defined (in local
coordinates) by functions of class $C^{l+\alpha}$. See, e.g.,
\cite{ladyzhenskaya1968linear} for precise definitions. \\

\section{Combined regular and impulse control}
 In this section we focus on the case of combined regular and impulse control
where the deviation $(X_t)$ from the target is given by 
\begin{equation*}
X_t = -X_t^\circ +\int_0^t u_s ds +  \sum_{j:0<\tau_j \leq t}\xi_j.
\end{equation*}

\subsection{Feedback strategies}
Motivated by various results in the literature and by the strategies
used by traders in practice, we consider a class of \emph{feedback
  strategies}. Under our asymptotic setting, an appropriate feedback strategy should depend on time. 
More precisely, let $(G_t)$ be a time-dependent random open bounded
domain in $\mathbb R^d$, $\xi_t$ be a time-dependent random function
such that $x\mapsto x+ \xi(x) \in C^0( \partial G_t , G_t)$, and
$(U_t)$ be a time-dependent bounded random function such that $U_t\in
C^0(\overline{G}_t, \bbR^d)$.  We say that the triplet $(U_t, {G}_t,
\xi_t)$ is continuous if $(U^g_t, \partial{G}_t, \xi^g_t)$ is
continuous as closed set-valued processes w.r.t.~the Hausdorff
distance $H$ and that  $(U_t, G_t, \xi_t)$ is progressively measurable, if $(U^g_t, \partial{G}_t, \xi^g_t)$ is progressively measurable w.r.t.~$(\calF_t)$ as closed set-valued processes (see \cite{kisielewicz2013stochastic} for more details).\\

\noindent \emph{In the sequel we will consider $(U_t, G_t,\xi_t)$ which is continuous and progressively measurable. } Intuitively, we require that the data $(U_t, G_t,  \xi_t)$ is determined in a non-anticipative way, based on information up to time $t$, and does not vary too much in time.  Note that since $\partial G_t$ is continuous, the topology of the domain $G_t$ (number of holes, boundedness, etc. ) remains the same. \\

Given a triplet $(U_t,G_t,\xi_t)$, we consider a family of stopping
times $(\tau_k)_{k\geq 0}$ and a family of processes $(\widehat
X^{(k)}_t)^{k\geq 0}_{\tau_k \leq t\leq \tau_{k+1}}$ satisfying 
\begin{align*}
&\tau_0 = 0,\quad \tau_{k+1} = \inf\{t>\tau_k: \widehat
  X^{(k)}_t\notin G_t\}\quad \text{for} \quad k\geq 0,\\
&
\widehat X^{(0)}_0 = 0,\quad d\widehat X^{(k)}_t = - dX^{\circ}_t + U_t(\widehat X^{(k)}_t) dt,\quad
\widehat X^{(k+1)}_{\tau_{k+1}} = \widehat X^{(k)}_{\tau_{k+1}} +
\xi_{\tau_{k+1}}(\widehat X^{(k)}_{\tau_{k+1}})  \quad \text{for}\quad
k\geq 0. 
\end{align*}
The existence of such processes will be discussed below. 
We now define the controlled deviation process by 
$$
X^{(U,G,\xi)}_t = \sum_{k\geq 0} \widehat X^{(k)}_t\mathbf
1_{\tau_k \leq t< \tau_{k+1}}
$$
and the corresponding boundary chain by 
$$
Y^{(U,G,\xi)}_j = \widehat X^{(j)}_{\tau_j},\quad j\geq 1. 
$$

\noindent In other words, when the deviation $X_t$ is inside the domain $G_t$,
only the regular control is active and the tracker maintains a speed
of $U_t(X_t)$. When $X_{t-}$ touches the boundary $\partial G_t$ at time
$\tau$, a jump of size $\xi_\tau(X_{\tau-})$ towards the interior of $G_t$
takes place and the process $X_t$ takes the value $X_{\tau-} +
\xi(X_{\tau-})$ at time $\tau$. We then have
\begin{equation}
X^{(U, G, \xi)}_t = -X^\circ_t + \int_0^t U_s(X^{(U, G, \xi)}_s)ds + \sum_{ 0<\tau^{(U, G, \xi)}_j \leq t}\xi_{\tau_j}(Y^{(U, G, \xi)}_j).\label{controlx}
\end{equation}
and we define $u_t = U_t(X^{(U, G, \xi)}_t)$ and $\xi_j = \xi_{\tau_j}(Y^{(U, G, \xi)}_j)$.
\\

\noindent Recall that the asymptotic framework introduced in
 \cite{cai2015asymptotic} consists in considering the
 sequence of optimization problems indexed by $\varepsilon\to0$ with
 optimization functionals
\begin{equation}\label{eqn: cost eps}
J^\varepsilon(u^\varepsilon, \tau^\varepsilon, \xi^\varepsilon) = \int_0^T (r_t D(X_t^\varepsilon) + \varepsilon^{\beta_Q}l_t Q(u^\varepsilon_t)) dt + \sum_{j: 0 <\tau^\varepsilon_j \leq T}( \varepsilon^{\beta_F}k_{\tau^\varepsilon_j}F(\xi_j^\varepsilon) + \varepsilon^{\beta_P}h_{\tau^\varepsilon_j}P(\xi_j^\varepsilon)),	
\end{equation}
where
\begin{equation*}
X_t^\varepsilon = -X_t^\circ +\int_0^t u^\varepsilon_s ds +  \sum_{j:0<\tau^\varepsilon_j \leq t}\xi_j^\varepsilon,
\end{equation*}
and $\beta_Q$, $\beta_F$ and $\beta_P$ are real numbers such that 
$$
\frac{\beta_F}{\zeta_D+2-\zeta_F} = \frac{\beta_P}{\zeta_D+2-\zeta_P}
= \frac{\beta_Q}{\zeta_D+ \zeta_Q} = \beta
$$
for some $\beta>0$. \\

\noindent Given a triplet $(U_t,G_t,\xi_t)$, we define $U^\varepsilon_t(x)  =
{\varepsilon^{-(\alpha-1)\beta}}U_t(\varepsilon^{-\beta}x)$,
$\xi_t^\varepsilon(x) = \varepsilon^{ \beta} \xi_{t}({\varepsilon^{-
    \beta}} x)$  and $G^{\varepsilon}_t = \varepsilon^{\beta} G_t$,
and construct the sequence of controlled processes and feedback strategies $(X^\varepsilon,
u^\varepsilon, \tau^\varepsilon, \xi^\varepsilon)$ as in \eqref{controlx}.
We make the following assumption.
\begin{asmp}\label{asmp: existence}
The controlled deviation process and the feedback strategy $(X^\varepsilon, u^\varepsilon, \tau^\varepsilon, \xi^\varepsilon)$ exist and are unique 
for each $\varepsilon >0$.
\end{asmp}

\noindent A rigorous verification of the above definition requires detailed
analysis with specific conditions on $(U_t,G_t, \xi_t)$. We  now describe a simple situation where the above assumption is verified. 
\begin{lem}\label{existence}
Let $(U_t)$ be locally Lipschitz, and assume that there exists a
\emph{potential function} $V$ for the jump rule, i.e. 
\begin{align}\label{eqn: bound K}
&|U_t(x) - U_t(y)|\leq K_t |x- y|, \quad \forall x, y\in \bbR^d, \forall (t, \omega), \\
&V(x+ \xi_t(x)) - V(x)<0, \quad \forall (t, \omega),\label{potential}
\end{align}
where $K_t$ is locally bounded and  $V \in C^2(\mathbb R^d,\mathbb
R)$. 
Then Assumption \ref{asmp: existence} holds.  
\end{lem}

\begin{proof}
Modulo a localization procedure, we may assume that $K_t$, $b_t$ and $a_t$ are bounded
and $V(x+ \xi_t(x)) - V(x)<-\delta <0$ for all $(t,\omega)$ and all
$x\in \partial G_t$, where $\delta$
is a constant. Similarly, the function $V$ as well as its first and second derivatives may
be assumed to be bounded when their argument belongs to $G_t$. \\

\noindent For the existence of $\widehat X^{(k)}$ for each $k$ (in the strong sense), Lipschitz-type regularity on $U_t$ is sufficient (see \cite[Chapter V]{protter2004stochastic} for more details). For the existence of $X^\varepsilon$ on the whole horizon $[0, T]$, it suffices to show that
\begin{equation*}
\lim_{j\to \infty} \tau^\varepsilon_j = +\infty 
\end{equation*}
almost surely. 
By It\^o's lemma,
\begin{align*}
V(X^\varepsilon_{\tau^\varepsilon_n}) - V(0) &=
\int_0^{\tau^\varepsilon_n}
\Big\{(U(X^\varepsilon_t)-b_t)\nabla V(X^\varepsilon_t)
  +\frac{1}{2}\sum_{ij}{a_{ij,t}}\partial^2_{ij} V (X^\varepsilon_t) \Big\}dt \\&-
\int_0^{\tau^\varepsilon_n} \nabla V(X^\varepsilon_t)^T \sqrt{a_t} dW_t + \sum_{k=1}^n
\{V(X^\varepsilon_{\tau^\varepsilon_k}) - V(X^\varepsilon_{\tau^\varepsilon_{k}-})\},
\end{align*}
so that 
$$
\delta n \leq C(1+ \tau^\varepsilon_n) + \int_0^{\tau^\varepsilon_n}
\nabla V(X^\varepsilon_t)^T \sqrt{a_t} dW_t
$$
for all $n\geq 1$, 
which is only possible if $\tau_n^\varepsilon \to +\infty$ a.s.

\end{proof}

\subsection{Asymptotic performance}

In order to have well-behaved strategies, we restrict ourselves to the following class of admissible triplets $(U_t, G_t, \xi_t)$.

\begin{dfn}[Admissible Strategy] The triplet $(U_t, G_t, \xi_t)$ is
  said to be admissible if the following conditions hold true. 
\begin{enumerate}

\item
(Potential) There exists $V\in C^2(\bbR^d, \bbR)$ such that
\eqref{potential}
 is satisfied. 
\item 
(Separability) For any $(t, \omega)$, let $U = U^\omega_t$,
$G=G_t^\omega$ and $\xi= \xi_t^\omega$. Then there exists a unique couple $(\pi, \nu)\in \scrP(\overline{G}) \times \scrM(\partial G)$ verifying the constraints
\begin{equation}\label{eqn: separability}
\int_{\overline{G}}A_U^af(x, u)\pi(dx) + \int_{ \partial G}B_\xi f(x)\nu(dx) = 0, \quad \forall f\in C^2_0(\bbR^d),
\end{equation}
where 
\begin{equation*}
A_U^{a}f(x) = \frac{1}{2}\sum_{i, j}a_{ij}\partial^2_{ij} f(x) + {U(x)\cdot \nabla f(x)}, \quad B_\xi f(x) = f(x+\xi(x)) - f(x).
\end{equation*} 
We note $\pi =: \pi^{(a, U, G, \xi)}$ and $\nu =: \nu^{(a, U, G,\xi)}$ to indicate the dependence on $(a, U, G, \xi)$. 
\end{enumerate}
\end{dfn}

\noindent The separability assumption is related to the existence and uniqueness
of a stationary distribution for the diffusion with jumps from the
boundary. The following lemma gives an example of the situation where
this assumption holds.

\begin{lem}\label{separ.jump}
Let $(a_{ij})$ be positive definite, let $G$ be a connected bounded domain of
$\mathbb R^d$ of class $C^{2+\alpha}$ for some $\alpha>0$, let $U$ be
Lipschitz and let $\xi: \partial G \to G$ be continuous.  Assume that
there exists a function $V\in C^2(\mathbb R^d)$ such that
\eqref{potential} is satisfied. 
Then there exists a unique couple $(\pi, \nu)\in
\scrP(\overline{G}) \times \scrM(\partial G)$ verifying the constraints \eqref{eqn: separability}.
\end{lem}
\begin{proof}
Let $\overline X$ denote the controlled deviation process with dynamics
\begin{align}
d\overline X_t = \sqrt{a} dW_t + U(\overline X_t) dt,\label{dyn}
\end{align}
on $G$ and jumps from the boundary $\partial G$ defined by the mapping $\xi$.
This process is defined in the same way as \eqref{controlx}.  We also
denote by $(\bar \tau_j)_{j\geq 1}$ the corresponding jump times and
by $(\overline Y_j)_{j\geq 1}$ the corresponding boundary chain. By
continuity, we may assume that $V(x+ \xi(x)) - V(x)<-\delta <0$ for
all $x\in \partial G$, where $\delta$ is a constant. \\

\noindent For measurable subsets $\mathcal A\subset \overline G$,
$\mathcal B\subset
\partial G$, define, for all $t>0$,
$$
\pi_t(\mathcal A) = \frac{1}{t}\mathbb E\Big[\int_0^t \mathbf 1_{\mathcal
  A}(\overline X_s) ds\Big],\qquad
\nu_t(\mathcal B) = \frac{1}{t}\mathbb E\Big[\sum_{j=1}^{N_t} \mathbf
1_{\mathcal B}(\overline X_{\bar\tau_j-})\Big], 
$$
where $N_t = \sum_{j\geq 1}\mathbf 1_{\bar\tau_j\leq t}$. 
By It\^o's lemma, 
\begin{align*}
V(\overline X_{t}) - V(0) &=
\int_0^{t}
\Big\{U(\overline X_s)\nabla V(\overline X_s)
  +\frac{1}{2}\sum_{ij}{a_{ij}}\partial^2_{ij} V (\overline X_s) \Big\}ds \\&-
\int_0^{t} \nabla V(\overline X_s)^T \sqrt{a} dW_s + \sum_{k=1}^{N_t}
\{V(\overline X_{\bar\tau_k}) - V(\overline X_{\bar\tau_{k}-})\},
\end{align*}
so that there exists a constant $C<\infty$ such that  
$$
\delta N_t \leq C(1+ t) + \int_0^{t}
\nabla V(\overline X_s)^T \sqrt{a} dW_s\quad \text{and}\quad
\delta \mathbb E N_t \leq C(1+ t),
$$
which shows that  $(\nu_t)$ is tight. On the other hand, $(\pi_t)$ is tight as a
family of probability measures on a bounded set. Therefore, there
exist measures $\pi^* \in \mathcal P(\overline G)$ and $\nu^*\in
\mathcal M(\partial G)$ and a sequence $(t_n)_{n\geq 1}$ converging to
infinity such that $\nu_{t_n}\to \nu^*$ and $\pi_{t_n}\to \pi^*$
weakly as $n\to\infty$. \\

\noindent Let $f\in C^2_0(\mathbb R^d)$. By It\^o's lemma, for every $n\geq 1$, 
$$
\mathbb E[f(\overline X_{t_n})] = \mathbb E[f(\overline X_0)] -
t_n\int_{\overline G} A^a_U f(x) \pi_{t_n}(dx) - t_n\int_{\partial G}
B_\xi f(x) \nu_{t_n}(dx).
$$
Dividing by $t_n$ on both sides and making $n$ tend to infinity, 
we conclude that 
$$
\int_{\overline G} A^a_U f(x) \pi^*(dx) +\int_{\partial G} Bf(x)
\nu^*(dx) = 0,
$$
which proves existence of the couple $(\pi,\nu)$ satisfying \eqref{eqn: separability}.\\

\noindent  Let us now show that this couple is unique.
By classical results on elliptic equations (see e.g., \cite[chapter 3]{ladyzhenskaya1968linear}),
 for every $g\in C^{2+\alpha}(\partial G)$,
there exists a unique solution $F^g \in C^{2+\alpha}(\overline G)$ of
the equation $A^a_U F^g = 0$ on $\overline G$ with the boundary
condition $F^g = g$ on $\partial G$, which may be extended to a
$C^2_0(\mathbb R^d)$ function. Moreover, this solution may be
represented as follows: 
$$
F^g(x) = \mathbb E^x[g(\overline X_{\tau_1-})],
$$
where $\mathbb E^x$ denotes the expectation with initial value $x\in
G$. 
Substituting the function $F^g$ into \eqref{eqn:
  separability}, we obtain
\begin{align}
\int_{\partial G} g(x)\nu(dx) = \int_{\partial G} F^g(x+ \xi(x))\label{invmeasure}
\nu(dx).
\end{align}
Denote the transition kernel of the chain $(\overline Y_n)$ by $p(x,dy)$. Then, 
$$
F^g(x+\xi(x)) = \mathbb E[g(\overline Y_{n+1})|\overline Y_n = x] = \int_{\partial G} g(y)p(x,dy).
$$
Therefore, equation \eqref{invmeasure} (satisfied for all $g$) is
equivalent to 
$$
\nu(dx) = \int_{\partial G} \nu(dy) p(y,dx),
$$
which means that $\nu$ is uniquely defined up to a multiplicative
constant if and only if the Markov chain $(\overline Y_k)$ is
ergodic. \\

\noindent Let $G'\subset\subset G$ be a connected open domain such that
$x+\xi(x)\in G'$ for all $x\in \partial G$ (such a domain exists since
$\xi$ is continuous and $\partial G$ is closed) and choose $x_0 \in
G'$. Then, for a Borel subset $A\subset \partial G$, by Harnack's inequality,
$$
\max_{x\in \partial G} \int_A p(x,dy) = \max_{x\in \partial G} \mathbb
P[\overline X_{\tau_1-}^{x+\xi(x)}\in A] \leq \max_{x\in G'} \mathbb
P[\overline X_{\tau_1-}^{x}\in A] \leq C \mathbb P[\overline
X_{\tau_1-}^{x_0}\in A],
$$
for a constant $C<\infty$ which does not depend on $A$. Therefore, the
Markov chain $(Y_k)$ satisfies Doeblin's condition (see e.g.,
\cite{tweedie1975sufficient}), which implies ergodicity. This shows
the uniqueness of the measure $\nu$ up to a multiplicative
constant. To show that $\nu$ is unique, consider now a function $f \in
C^2_0$ such that $A^a_U f = 1$ on $\overline G$. \\

\noindent Finally, to show the uniqueness of $\pi$, one can choose a function
$f$ to be the unique solution in $C^{2+\alpha}(\overline G)$ of the
equation $A^a_U f = g$, where $g$ is an arbitrary function in
$C^\alpha(\overline G)$. 

\end{proof}





\begin{ex}\label{ex: jump}
Let $A_t$ be a continuous adapted process with values in $\calS^d_+$, 
and let $a_t$ be a continuous adapted process taking values in $[0,1)$. Define
$G_t = \{x\in \bbR^d: x^TA_tx <1\}$ and $\xi_t(x)= (1-\alpha_t)
x$. Then, condition \eqref{potential} is satisfied with $V(x) =
\|x\|^2$. 

\end{ex}



\noindent Now we state the first main result in this paper. 

\begin{theo}[Asymptotic performance for combined regular and impulse control]\label{theo: limit in proba}
Consider  an admissible triplet $(U_t, G_t, \xi_t)$, suppose that Assumption
\ref{asmp: existence} is satisfied and let $(X^\varepsilon, u^\varepsilon, \tau^\varepsilon,\xi^\varepsilon)$ be the corresponding feedback strategy. Then,
\begin{equation*}
 \frac{1}{\varepsilon^{\zeta_D \beta}} J^\varepsilon(u^\varepsilon, \tau^\varepsilon, \xi^\varepsilon) \to_p \int_0^T c(a_t, U_t, G_t, \xi_t; r_t,l_t, k_t, h_t) dt,
\end{equation*}
where $c(a, U, G, \xi; r, l, k, h)$ is given by 
\begin{align*}
c(a, U, G, \xi; r, l, k, h)& =
 \int_{\overline{G}} [ r D(x) + l Q (U(x))]\pi^{(a, U, G, \xi)}(dx) \\
&\qquad + \int_{\partial G} [k F(\xi(x)) + h P(\xi(x))] \nu^{(a,U, G, \xi)}(dx).
\end{align*}
Moreover, the convergence holds term by term for the cost functions $D$, $Q$, $F$, $P$ respectively. 
\end{theo}




\begin{rmk}
In \cite{cai2015asymptotic} we have established that for any sequence
of 
admissible strategies $(u^\varepsilon,\tau^\varepsilon,\xi^\varepsilon)$,
\begin{align}
\liminf_{\varepsilon \to 0}\frac{1}{\varepsilon^{\zeta_D \beta}}
J^\varepsilon(u^\varepsilon, \tau^\varepsilon, \xi^\varepsilon) \geq_p
\int_0^T I(a_t, r_t, l_t, k_t, h_t) dt\label{lowerbound}
\end{align}
where $I= I(a, r,l, k, h) $ is given by
\begin{equation}\label{eqn: lower bound}
I = \inf_{(\mu, \rho)} \int_{\bbR^d \times \bbR^d}(r D(x) + l Q(u) )\mu(dx \times du) + \int_{\bbR^d\times \bbR^d}(kF(\xi) + h P(\xi))\rho(dx\times d\xi),
\end{equation}
with $(\mu, \rho)\in\scrP(\bbR^d\times \bbR^d)\times \scrM(\bbR^d\times \bbR^d )$ verifying the following constraints
\begin{equation*}\label{LP: constraint}
\int_{\bbR^d\times \bbR^d}A^af(x, u)\mu(dx\times du) + \int_{\bbR^d\times \bbR^d}Bf(x, \xi)\rho(dx\times d\xi) = 0, \quad \forall f\in C^2_0(\bbR^d),
\end{equation*}
where
$$
A^af(x,u) = \frac{1}{2} \sum_{i,j}a_{ij} \partial^2_{ij} f(x) + u
\cdot f(x), \qquad B f(x,\xi) = f(x+\xi)-f(x). 
$$
If we can find $(U_t, G_t, \xi_t)$ such that $c(a_t, U_t, G_t, \xi_t;
r_t; l_t, k_t, h_t) = I(a_t, r_t, l_t, k_t, h_t)$ for all $t$, then the lower bound of \cite{cai2015asymptotic} is sharp and we
say that the feedback strategy $(U_t, G_t, \xi_t)$ is
\emph{asymptotically optimal}. In particular, in the one-dimensional
case, when $D(x) = x^2$, $Q(u) = u^2$, $F(\xi) = 1$ and $P(\xi) =
|\xi|$, the lower bounds found in Examples 4.5 and 4.6 of \cite{cai2015asymptotic} correspond to
strategies of feedback form, which means that these bounds are sharp
and the asymptotically optimal strategies are explicit. 
\end{rmk}
\begin{ex}
In this example we revisit the problem of optimal tracking in the
multidimensional setting with fixed costs. A very similar problem was studied in \cite{gobet2012almost}, in the context of optimal rebalancing
of hedging portfolios, and an asymptotically optimal strategy based on
hitting times of ellipsoids has been
found. \\ 

\noindent Let $D(x) = x^T\Sigma^D x$, $l=0$, $F(\xi) \equiv 1$ and
$P(\xi) \equiv 0$. Without loss of generality we set $r=1$ and $k=1$. 
The cost functional of the
linear programming problem becomes, 
$$
I = \inf_{(\mu, \rho)}  \int_{\bbR^d }x^T\Sigma^D x\, \mu(dx) + \rho(\bbR^d\times \bbR^d),
$$
where $\mu$ is now a probability measure on $\bbR^d$ only, since
there is no regular control. 
By Lemma 3.1 in \cite{gobet2012almost}, the matrix equation
\begin{align}
2  (a^{1/2} B a^{1/2})^2     + a^{1/2} B a^{1/2}\,
\mathrm{Tr}(a^{1/2} B a^{1/2}) = 2 a^{1/2}\Sigma^D a^{1/2} \label{eqgobet}
\end{align}
admits a unique solution $B\in \mathcal S^d_+$. We define
$$
w(x) =  \left\{\begin{aligned}&x^T Bx - \frac{ (x^T B x)^2}{4}, &&x^TB
    x < 2\\
&1,&& x^T B x\geq 2.
\end{aligned}\right.
$$
Clearly, $w\in C^1(\mathbb R^{d})\cap C^2(\mathbb R^d \setminus N)$,
where $N = \{x: x^T Bx = 2\}$. Note also that 
\begin{align}
\frac{1}{2}\sum_{ij} a_{ij} \partial^2_{ij} w(x) = \left\{\begin{aligned}&\mathrm{Tr}(aB)\Big(
1-\frac{x^T B x}{2} \Big) -  x^T B a B
x , &&x^TB
    x < 2\\ &0, && x^TB
    x > 2\end{aligned}\right. \label{2der}
\end{align}
We are going to apply Lemma 8.1 of
\cite{cai2015asymptotic} to compute a lower bound for $I$. Property
2 of Lemma 8.1, is a straightforward consequence of
\eqref{2der}. Property 3 of this lemma is
also easily deduced from \eqref{2der}, with $I^V = \text{Tr}(aB)$. It
remains to check Property 1, which requires that $\mu(N) = 0$ for all $(\mu,\rho)$ with $\int x^T \Sigma^D x \, \mu(dx)<\infty$ and  
\begin{align}
\int_{\mathbb R^d_x}\frac{1}{2}\sum_{ij} a_{ij} \partial^2_{ij} f(x) \mu(dx) +
\int_{\mathbb R^d_x \times \mathbb R^d_\xi}(f(x+\xi)-f(x))\rho(dx, d\xi) = 0,\quad f\in C^2_0(\mathbb R^d).\label{constraint}
\end{align}
Assume that this is not the case.  For $\delta \in (0,1)$ let $\phi_\delta: \mathbb R^d \to
[0,1]$ be defined as follows. 
\begin{align*} 
&\phi_\delta = 0,\quad &&x^T B x < 2-2\delta\\ 
&\phi_\delta = (x^T B x - 2+2\delta)/\delta,\quad &&2-2\delta \leq x^T B x <
  2-\delta\\
&\phi_\delta = 1,\quad &&2-\delta \leq x^T B x < 2+\delta\\
&\phi_\delta =(2+2\delta - x^TBx)/\delta,\quad &&2+\delta \leq x^T B x < 2+2\delta\\
&\phi_\delta = 0,\quad &&2+2\delta \leq x^T B x .
\end{align*}
Let $f_\delta$ be the extension to $C^2_0(\mathbb R^d)$ of the solution of 
$$
\frac{1}{2}\sum_{ij} a_{ij} \partial^2_{ij} f_\delta(x)  =
\phi_\delta(x),\quad x^T B x \leq 4
$$
with boundary condition $f_\delta (x) = 0$ for $x^T B x = 4$. By
Feynmann-Kac formula,
$$
f_\delta(x) = \mathbb E \Big[\int_0^{\tau} \phi_\delta(\Xi^x_s) ds\Big],
$$
where $\Xi^x$ is defined by.
$$
\Xi^x_t = x + \int_0^t \sqrt{a} dW_s,
$$
and $\tau = \inf\{t>0: (\Xi^x_t)^T B\, \Xi^x_t=4\}$. 
By the dominated convergence theorem, $f_\delta(x) \to 0$ as $\delta
\to 0$ for all $x$ with $x^T B x \leq 4$, and also uniformly on $x$. On the other hand,
$$
\int_{\mathbb R^d_x}\frac{1}{2}\sum_{ij} a_{ij} \partial^2_{ij} f_\delta(x) \mu(dx) =
\int_{\mathbb R^d_x}\phi_\delta(x) \mu(dx)\geq \mu(N),
$$
which yields a contradiction with \eqref{constraint}. Finally, Lemma 8.1 of
\cite{cai2015asymptotic} may be applied and yields that $I\geq
\text{Tr} (aB)$.

\noindent Now consider $G:= \{x\in \mathbb R^d: x^T B x < 2\}$ and $\xi(x) =
-x$. Since there is no regular control, we let $U(x)=0$ and omit this
variable as well as the variables $l$ and $h$ below, and we also set
$r=1$ and $k=1$ as above. Then,
$$
c(a,G,\xi) = \int_{\overline G} D(x) \pi^{(a,G,\xi)} (dx) +
\nu^{(a,G,\xi)}(\partial G), 
$$
where the measures $\pi^{(a,G,\xi)} $ and $\nu^{(a,G,\xi)}$ are
uniquely defined by Lemma \ref{separ.jump}. Since the measure
$\pi^{(a,G,\xi)}$ does not charge the set $N$, we may show by an
approximation argument that 
$$
\int_{G} A^a w(x) \pi^{(a,G,\xi)}(dx) + \int_{\partial G}B_\xi w(x)
\nu^{(a,G,\xi)}(dx) = 0.
$$
Substituting the expressions of $A^a$ and $B_\xi$, and using
\eqref{2der}, we finally obtain
$$
\int_G x^T \Sigma^D x\, \pi^{(a,G,\xi)}(dx) + \nu^{(a,G,\xi)}(\partial G) = \text{Tr}(aB) .
$$
To summarize, we have established that for the cost functional of the form
$$
J^\varepsilon(\tau^\varepsilon,\xi^\varepsilon) = \int_0^T r_t
\,(X^\varepsilon_t)^T \Sigma^D  X^\varepsilon_t\, dt +
\varepsilon\sum_{j: 0< \tau^\varepsilon_j \leq T}
k_{\tau^\varepsilon_j},
$$ 
and for any admissible control strategy $(\tau^\varepsilon,\xi^\varepsilon)$, we have a lower bound
$$
\liminf_{\varepsilon \to 0}\varepsilon^{-1/2}
J^\varepsilon(\tau^\varepsilon, \xi^\varepsilon) \geq_p
\int_0^T I(a_t,  r_t, k_t) dt,
$$
where 
$$
I(a_t,r_t,k_t) = \text{Tr}(a_t B_t)\sqrt{r_t k_t},
$$
where $B_t$ is the solution of Equation \eqref{eqgobet} with
matrix $a_t$. This lower bound is sharp and may be
attained by a feedback strategy based on hitting times of the
time-varying domain $G_t = \{x\in \mathbb R^d: x^T B_t x <
2\sqrt{\frac{k_t}{r_t}}\}$. Note that unlike the asymptotically
optimal strategy found in \cite{gobet2012almost}, the
domain is not fixed between the consecutive hitting times. 


\end{ex}

\section{Extensions to other types of control}\label{sec: extensions}

\subsection{Combined regular and singular control}
In the absence of fixed cost component, we consider a family of strategies $(u^\varepsilon, \gamma^\varepsilon, \varphi^\varepsilon)$ with $u^\varepsilon$ a progressively measurable process, $\gamma^\varepsilon_t\in \Delta = \{\gamma \in \bbR^d| \sum_i |\gamma^i| = 1\}$ and $\varphi^\varepsilon_t$  non-decreasing such that
\begin{equation*}
X^\varepsilon_t = -X^\circ_t + \int_0^tu^\varepsilon_s ds + \int_0^t \gamma^\varepsilon_{s}d\varphi^\varepsilon_s. 	\end{equation*}
The associated cost functional is given by
\begin{equation*}
J^\varepsilon(u^\varepsilon, \gamma^\varepsilon, \varphi^\varepsilon) = \int_0^T(r_tD(X^\varepsilon_t) + \varepsilon^{\beta_Q}l_tQ(u^\varepsilon_t))dt + \int_0^T \varepsilon^{\beta_P}h_tP(\gamma^\varepsilon_{t})d\varphi^\varepsilon_t.	
\end{equation*}

\noindent The corresponding class of feedback strategies will involve continuous
controls of local-time type. More precisely, let $G_t$ be a moving domain with piecewise $C^2$ boundary and $U_t$ a continuous function defined on $\overline{G}_t$ as before, let $\Gamma_t$ be a  set-valued mapping defined on $\partial G_t$ with \emph{closed graph}  such that $\Gamma_t(x)$ is a non-empty closed convex cone in $\bbR^d$ with vertex at the origin $0$ for each $x\in \partial G_t$. Usually $\Gamma_t(x)$ contains only a single direction on the smooth part of $\partial G$. If $\partial G$ is $C^2$ then $\Gamma_t(x)$ can also be expressed in terms of a vector field on $\partial G$.  \emph{We assume that the triplet $(U_t, {G}_t, \Gamma_t)$ is continuous and progressively measurable.} \\

\noindent The feedback strategy $(X^\varepsilon, u^\varepsilon,
\gamma^\varepsilon, \varphi^\varepsilon)$ based on $(U_t, G_t,
\Gamma_t)$ is defined by
\begin{equation*}
X^\varepsilon_t= -X^\circ_t + \int_0^t u^\varepsilon_s ds + \int_0^t \gamma^\varepsilon_sd\varphi^\varepsilon_s,
\end{equation*}
with $\gamma^\varepsilon_t\in \Gamma_t$, $\varphi^\varepsilon$ \emph{continuous} non decreasing  and 
\begin{equation*}
u^\varepsilon_t = \varepsilon^{-(\alpha-1)\beta}U_t(\varepsilon^{-\beta}X^\varepsilon_t),
\end{equation*}
such that
\begin{enumerate}
\item
 ${\varepsilon^{-\beta}}X^\varepsilon_t \in \overline{G}_t$,
\item
$ \gamma^\varepsilon_t\in \Gamma_t(X^\varepsilon_t)\cap \Delta$, $d\varphi_t$-almost surely, 
\item 
$\int_0^T \mathbbm{1}_{G_t}({\varepsilon^{-\beta}}X^\varepsilon_t) d\varphi^\varepsilon_t = 0. 	$
\end{enumerate}

\noindent Note that the condition $\gamma_t^\varepsilon\in \Delta$
allows $\gamma^\varepsilon_t$ and $\varphi^\varepsilon_t$ to be
uniquely defined.  We assume once again existence of the strategy. 
\begin{asmp} \label{asmp: existence singular}
The feedback strategy $(X^\varepsilon, u^\varepsilon, \gamma^\varepsilon, \varphi^\varepsilon)$ exists and is unique 
for each $\varepsilon >0$.
\end{asmp}
\noindent The existence of such strategies is closely related to the Skorohod
oblique reflection problem in time-dependent domains. While it is easy
to establish in dimension one (see
\cite{slominski2013stochastic}), it is not at all trivial in higher
dimension. We refer to \cite{nystrom2010skorohod} for precise
sufficient conditions for the existence of such strategy. \\



\noindent From now on, we restrict ourselves to $(U_t, G_t, \Gamma_t)$ for which the above strategies exist and are well-behaved. 
\begin{dfn}[Admissible Strategy]
The triplet $(U_t, G_t, \Gamma_t)$ is admissible if the following
conditions hold. 
\begin{enumerate}

\item 
(Potential.) There exists $V\in C^2(\bbR^d)$ such that $\forall (t, \omega)$, 
\begin{equation}
\crochet{\nabla V(x), \gamma} <0, \quad \forall x\in \partial G^\omega_t, \gamma\in \Gamma^\omega_t(x). \label{potential.reflect}
\end{equation}
We say that $V$ is a potential function for $(G_t, \Gamma_t)$.
\item 
(Separability.) For any $(t, \omega)$, let $U=U^\omega_t, G=G^\omega_t$ and $\Gamma = \Gamma_t^\omega$, then there exists a unique couple $(\pi, \rho)\in \scrP(\overline{G})\times \scrM(\Gamma^g_\Delta)$ such that
\begin{equation}\label{eqn: char stn mes}
\int_{\overline{G}}A_U^af(x) \pi(dx) + \int_{\Gamma^g_\Delta} Bf(x, \gamma)\rho(dx, d\gamma)=0, \quad \forall f\in C^2_0(\bbR^d), 
\end{equation}
where $\Gamma^g_\Delta = \{(x, \gamma)|x\in\partial G,  \gamma \in \Delta\cap \Gamma(x) \}$ and
\begin{equation*}
A_U^{a}f(x) = \frac{1}{2}\sum_{ij}a_{ij}\partial^2_{ij}f(x) +U(x)\cdot \nabla f(x), \quad Bf(x, \gamma) = \gamma\cdot \nabla f(x).
\end{equation*} 
We note $\pi=:\pi^{(a, U, G, \Gamma)}$ and $\rho=:\rho^{(a, U, G, \Gamma)}$.
\end{enumerate}
\end{dfn}

\noindent The following lemma clarifies the meaning of the measures $\pi$ and
$\rho$ and provides an example of a situation where the
separability assumption holds. For simplicity, we consider a smooth
boundary and place ourselves in the classical setting of
Stroock and Varadhan \cite{stroock1971diffusion}. 

\begin{lem}\label{separ.reflect}
Let $G$ be a connected bounded open subset of $\mathbb R^d$. Let
$(a_{ij})$ be positive definite, $u:G \to \mathbb R^d$ be bounded and
measurable, and $\gamma: \partial G \to \mathbb R^d$ be bounded and
continuous. Assume that there exists a function $V \in C^2(\mathbb
R^d)$ such that $G =
\{x\in \mathbb R^d: V(x)<0\}$ and $\gamma(x)\cdot \nabla V(x)<0$ for
all $x\in \partial G$. Define $\Gamma(x) = \{s\gamma(x),s\geq
0\}$. 
Then there exists a unique couple $(\pi, \nu)\in
\scrP(\overline{G}) \times \scrM(\partial G)$ verifying the constraints \eqref{eqn: char stn mes}.
\end{lem}
\begin{proof}
 Throughout the proof we shall assume with no loss of generality that 
$|\nabla V(x)|>1$ and $\gamma(x)\cdot
\nabla V(x)<-\beta <0$ for some $\beta>0$ for $x\in \partial G$, and
that $V$, as well as its first and second derivative are bounded. 
Also, under our assumptions, condition \eqref{eqn: char stn mes}
becomes
$$
\int_{\overline{G}}A_U^af(x) \pi(dx) + \int_{\partial G} Bf(x)\rho(dx)=0, \quad \forall f\in C^2_0(\bbR^d), 
$$
with $Bf(x) = \gamma(x)\cdot\nabla f(x)$.\\

\noindent Let $(\mathbb P, X)$ be a solution to
the submartingale problem with coefficients $a$, $U$ and $\gamma$ (see
Theorem 3.1 in \cite{stroock1971diffusion}), and let $\xi$ be the
associated non-decreasing process which increases only on the boundary
$\partial G$. For measurable subsets $\mathcal A\subset \overline G$,
$\mathcal B\subset
\partial G$, define, for all $t>0$,
$$
\pi_t(\mathcal A) = \frac{1}{t}\mathbb E\int_0^t \mathbf 1_{\mathcal A}(X_s) ds,\qquad
\rho_t(\mathcal B) = \frac{1}{t}\mathbb E\int_0^t \mathbf 1_{\mathcal B}(X_s) d\xi_s. 
$$
By Theorem 2.5 in \cite{stroock1971diffusion}, 
$$
V(X_t) - \int_0^t \mathbf 1_G(X_s) A^a_U V(X_s) ds  - \int_0^t \nabla
V\cdot \gamma d\xi_s 
$$
is a martingale. Taking the expectation yields
$$
\mathbb E[V(X_t)] = \mathbb E[V(X_0)] - t\int_{\overline G} \mathbf
1_G(x) A^a_U V(x) \pi_t(dx) - t\int_{\partial G} \nabla V(x) \cdot
\gamma(x) \rho_t(dx).
$$ 
Using the boundedness of
$U$, $V$ and the derivatives of $V$, and the condition that
$\gamma(x)\cdot \nabla V(x)<-\beta$ we conclude that 
$$
\rho_t(\partial G) \leq C
$$
for some constant $C<\infty$ which does not depend on $t$, which shows
that $(\rho_t)$ is tight. By the same argument as in the proof of Lemma
\ref{separ.jump}, using Theorems 2.4 and 2.5 in
\cite{stroock1971diffusion}, we conclude that there exists a couple $(\pi,\rho)$ satisfying \eqref{eqn: char stn mes}.\\


 \noindent Let us now prove that the couple $(\pi,\rho)$ satisfying \eqref{eqn:
   char stn mes} is unique. Choose one such couple $(\pi,\rho)$. By Theorem 1.7 in
 \cite{kurtz2001stationary}, there exists a process $X$ and a random
 measure $\Gamma$ with $\mathbb E[\Gamma(\cdot\times [0,t])] =
 t\rho(\cdot)$, such that $X$ is stationary and has distribution
 $\pi$, and for each $f\in C^2_0(\mathbb R^d)$, 
$$
f(X_t) - \int_0^t A^a_U f(X_s) ds - \int_{\overline  G\times [0,t] }
Bf(x) \Gamma(dx\times ds) 
$$
is a martingale. Moreover, from the discussion in section 2 of
\cite{kurtz2001stationary} we deduce that $\Gamma$ is a positive
measure, and hence that it is supported by $\partial G \times
[0,\infty)$.\\

\noindent Taking $f\in C^{2}_0(\mathbb R^d)$ satisfying $\gamma(x)\cdot\nabla
f(x)\geq 0$ for $x\in \partial G$, we see that $X$ is a stationary solution of the submartingale
problem (in the sense of \cite{stroock1971diffusion}), which means
that $\pi$ is necessarily an invariant measure of reflected diffusion
process in $G$. Since the generator of $X$ is strictly elliptic, and
the domain $G$ is bounded,
the invariant measure is unique\footnote{We thank Emmanuel Gobet for
  pointing out this reference.} (see \cite[pp.~97-98]{freidlin85functional}), proving the uniqueness of $\pi$. \\

\noindent It remains to show the uniqueness of $\rho$, or in other words to
prove that a signed measure $\rho$ satisfying
$$
\int_{\partial G} \gamma(x)\cdot\nabla f(x)\rho(dx)=0, \quad \forall f\in C^2_0(\bbR^d), 
$$
is necessarily null. To this end, choose $f(x) = \alpha(x) V(x)$,
with $\alpha \in C^{2}_0$. Then, $\nabla (\alpha V) = V \nabla \alpha
+ \alpha \nabla V = \alpha \nabla V$ on $\partial G$. Therefore,
$$
0 = \int_{\partial G} \alpha(x) \gamma(x)\cdot \nabla V(x)\rho(dx),
$$
which shows that $\rho(x) = 0$ since $\gamma(x) \cdot
\nabla V(x)<-\beta<0$ and $\alpha$ is arbitrary.

\end{proof}

\begin{ex}\label{ex: reflection}
Let $A_t$ be a continuous adapted process with values in $\mathcal
S^d_+$ , the set of symmetric positive definite matrices, define
$G_t = \{x\in \mathbb R^d: x^T A_t x <1\}$ and let $\gamma(x) = -x$. Then, assumptions of Lemma
\ref{separ.reflect} are satisfied with $V(x) = x^T A_t x -1$ and
\eqref{potential.reflect} holds with $V(x) = \|x\|^2$. 

%

\end{ex}


\noindent Now we are ready to state another main result. 
\begin{theo}[Asymptotic performance for combined regular and singular control]\label{theo: limit in proba singular}
Let  $(U_t, G_t, \Gamma_t)$ be an admissible triplet and let
Assumption \ref{asmp: existence singular} be satisfied. Then,
\begin{equation*}
 \frac{1}{\varepsilon^{\zeta_D \beta}} J^\varepsilon(u^\varepsilon, \gamma^\varepsilon, \varphi^\varepsilon) \to_p \int_0^T c(a_t, U_t, G_t, \Gamma_t; r_t,l_t, h_t) dt,
\end{equation*}
where $c(a, u, G, \Gamma; r, l, h)$ is given by 
\begin{align*}
c(a, u, G, \Gamma; r, l, h)& =
 \int_{\overline{G}}  (r D(x) + l Q (u(x))) \pi^{(a, u, G, \Gamma)}(dx) \\
&\qquad \qquad\qquad + \int_{\Gamma^g_\Delta}  h P(\gamma) \rho^{(a,u, G, \Gamma)}(dx\times d\gamma).
\end{align*}
Moreover, the convergence holds term by term for the cost functions $D$, $Q$ and $P$ respectively. 
\end{theo}
%


\begin{rmk}
In \cite{cai2015asymptotic} we have established a lower bound of
$J^\varepsilon$ of the form \eqref{lowerbound} with function $I= I(a, r,l, h) $ given by
\begin{equation}\label{eqn: lower bound}
I = \inf_{(\mu, \rho)} \int_{\bbR^d \times \bbR^d}(r D(x) + l Q(u)
)\mu(dx \times du) + \int_{\bbR^d\times \Delta \times \mathbb
  R^+_\delta}h P(\gamma)\rho(dx\times d\gamma\times d\delta),
\end{equation}
with $(\mu, \rho)\in\scrP(\bbR^d\times \bbR^d)\times
\scrM(\bbR^d\times \Delta \times \mathbb R^+_\delta )$ verifying the following constraints
\begin{equation*}\label{LP: constraint}
\int_{\bbR^d\times \bbR^d}A^af(x, u)\mu(dx\times du) +
\int_{\bbR^d\times \Delta\times \bbR^+_\delta}Bf(x,
\gamma,\delta)\rho(dx\times d\gamma\times d\delta) = 0, \quad \forall f\in C^2_0(\bbR^d),
\end{equation*}
where 
$$
Bf(x,\gamma,\delta) = \left\{\begin{aligned}&\langle \gamma,\nabla
    f(x)\rangle,&&\delta= 0,\\ &\delta^{-1}(f(x+\delta\gamma) - f(x)),&&\delta>0.\end{aligned}\right.
$$
Once again, in the one-dimensional
case, when $D(x) = x^2$, $Q(u) = u^2$ and $P(\xi) =
|\xi|$, the lower bounds found in Examples 4.4 and 4.7 of \cite{cai2015asymptotic} correspond to
strategies of feedback form, which means that these bounds are sharp
and the asymptotically optimal strategies are explicit. 
\end{rmk}

\subsection{The case when only one control is present}
The situations with singular or impulse control only are included in
the previous results. Only the case with regular control needs to be
treated separately since the domain in which the process $X$ evolves
is now unbounded. \\

\noindent Let $U_t: \bbR^d \to \bbR^d $ be a continuous predictable random
function. The feedback strategy $(X^\varepsilon_t, u^\varepsilon_t)$ based on the Markov control policy $U_t$ is given by
\begin{equation*}
X^\varepsilon_t =- X^\circ_t +u^\varepsilon_tdt, \quad u^\varepsilon_t = \varepsilon^{-\beta} U_t(\varepsilon^{-\beta}X^\varepsilon_t).
\end{equation*}
We assume that $U_t$, satisfies suitable conditions so that the
process $(X^\varepsilon,u^\varepsilon)$ exists and is unique for every
$\varepsilon>0$.\\


\noindent To study the asymptotic behavior of the controlled process $X^\varepsilon$, we consider the following class of admissible strategies. 

\begin{dfn}[Admissible strategy]
The feedback control $(U_t)$ is admissible if the following conditions
are satisfied. 
\begin{enumerate}
 \item
(Potential.) For each $(t, \omega)$, $U_t^\omega:\bbR\to \bbR$ is locally bounded and there exists a non-negative, inf-compact Lyapunov function $V\in C^2(\bbR^d)$ such that \begin{equation*}
A^{a_t}_{U_t}V(x) \leq \theta_t - 2 \Theta_t V(x), \quad \forall x\in \bbR_x,
\end{equation*}
with $\theta_t$, $\Theta_t$ positive processes. Moreover,  $rD+ l Q\circ U_t^\omega$ are dominated by $V$ near infinity, i.e. there exist locally bounded positive processes $R_t$ and $b_t$  such that
\begin{equation*}
rD(x) + lQ\circ u^\omega_t (x) \leq b^\omega_t V(x), \quad x\in \bbR^d\setminus B(x, R_t^\omega).
\end{equation*}
 \item
 (Separability.) For each $ (t, \omega)$, there exists a unique $\pi\in\scrP(\bbR^d)$ such that, 
 \begin{equation*}
\int_{\bbR^d}A^{a}_Uf(x)\pi(dx) = 0, \quad \forall f\in C_0^2(\bbR^d). 
\end{equation*}
We will denote $\pi = \pi^{(a, U)}$. 
\end{enumerate}
\end{dfn}
\begin{rmk}
Here  $\pi$ is the unique invariant measure of $dX_t = \sqrt{a} dW_t + U(X_t) dt$ under the feedback control $U$. 
\end{rmk}

\begin{ex}\label{ex: stoch}
Assume that $D$ and $Q$ are quadratic. Let $U_t(x)=  -\Sigma_t x$ with $\Sigma_t$ a continuous  process with values in $\calS_+^d$. Then $(U_t)$ is an admissible Markov policy.  To see this, it suffices to take $V(x) = x^T x$.  The separability of $\pi$ is equivalent to the ergodicity of $dX_t = \sqrt{a} dW_t -\Sigma X_t dt$.\\

\noindent In dimension one, let $D(x) = x^2$, $Q(u) = u^2$, then we have
\begin{equation*}
\pi(dx)= \frac{1}{\sqrt{2\pi \sigma }}e^{-\frac{x^2}{2 \sigma^2}} dx, \quad \sigma^2 = \frac{a}{2\Sigma}. 
\end{equation*}
\end{ex}

\noindent Below is our third main result. 

\begin{theo}[Asymptotic performance for regular control]\label{theo: limit in proba stoch}
Let $(u^\varepsilon_t)$ be the feedback strategy based on the admissible Markov control policy $U_t$, then,
\begin{equation*}
\lim_{\varepsilon\to 0} \frac{1}{\varepsilon^{\zeta_D \beta}} J^\varepsilon \to_p \int_0^T c(a_t,  U_t ; r_t, l_t) dt,
\end{equation*}
where $c(a, u  ; r,l)$ is given by 
\begin{equation*}
c(a, u; r, l) =
  \int_{\bbR^d} (r D(x) + l Q\circ u(x)) \pi(dx), 
\end{equation*}
with $\pi=\pi^{(a, u)}$.
\end{theo}
\begin{rmk}
In this case as well, we have a lower bound of the form
\eqref{lowerbound} with the function $I$ given by 
$$
I(a,r,l) = \inf_\mu \int_{\mathbb R^d_x \times \mathbb R^d_u} (rD(x) +
lQ(u))\mu(dx\times du)
$$
with $\mu \in \mathcal P(\mathbb R^d_x \times \mathbb R^d_u)$
satisfying the constraint
$$
\int_{\mathbb R^d_x \times \mathbb R^d_u} A^a f(x,u) \mu(dx\times du)
= 0\quad \forall f\in C^2_0(\mathbb R^d_x). 
$$
This shows that the bound given in Example 4.3 of
\cite{cai2015asymptotic} for the one-dimensional case with $D(x) =
x^2$ and 
$Q(u) = u^2$ is sharp.  Moreover, it can be generalized to the
multidimensional case as follows. \\

\noindent Let $D(x) = x^T D x$ and $Q(u) = u^T Q u$,
with $D, Q \in \mathcal S^+_d$, and consider the
Hamilton-Jacobi-Bellman equation
$$
\inf_u \left\{\frac{1}{2} \sum_{i,j=1}^d a_{ij} \frac{\partial^2
    w}{\partial x_i \partial x_j} + u^T \nabla w + l Q(u) + r D(x) -
  I^V\right\}=0,
$$
where $I^V$ is a constant which must be found as part of the
solution. It is easy to check that this equation admits an explicit
solution 
$$
w(x) = x^T G x \quad \text{with}\quad G = \sqrt{rl QD},\quad I^V = \frac{1}{2}\text{Tr}(aG),
$$
which corresponds to the feedback control 
$$
u(x) = -\sqrt{\frac{r}{l}Q^{-1}D} x.
$$
We then deduce that for any $\mu \in \mathcal P(\mathbb R^d_x \times
\mathbb R^d_u)$, 
\begin{align}
\int_{\mathbb R^d_x \times \mathbb R^d_u} (rD(x) +
lQ(u))\mu(dx\times du) \geq I^V - \int_{\mathbb R^d_x \times \mathbb
  R^d_u} A^a w(x,u) \mu(dx\times du). \label{estimate}
\end{align}
Since it is enough to consider measures $\mu$ for which 
$$
\int_{\mathbb R^d_x \times \mathbb R^d_u} (rD(x) +
lQ(u))\mu(dx\times du) < \infty,
$$
we can approximate the function $w$ with a sequence of functions in
$C^2_0(\mathbb R^d_x)$ and show that the integrals in the right-hand
side of \eqref{estimate} converge, proving that 
$$
\int_{\mathbb R^d_x \times \mathbb R^d_u} (rD(x) +
lQ(u))\mu(dx\times du) \geq I^V. 
$$
Now it remains to check that using the feedback strategy of
example \ref{ex: stoch} with $\Sigma = \sqrt{\frac{r}{l}Q^{-1}D}$
we have exactly $c(a,U;r,l) = I^V$, which shows that this strategy is
asymptotically optimal and the lower bound is sharp. 
\end{rmk}

\section{Proofs}\label{sec: proofs}

\paragraph{Proof of Theorem \ref{theo: limit in proba}}
Let $\varepsilon\mapsto T^\varepsilon$ be a positive decreasing
function such that $T^\varepsilon \to +\infty$ and $\varepsilon^2
T^\varepsilon \to 0$ as $\varepsilon\to 0$, and consider following the
rescaling of $X^\varepsilon$ over the horizon $(t, t+\varepsilon^2 T^\varepsilon]$:
\begin{equation*}
\wt{X}^{\varepsilon, t}_s = \frac{1}{\varepsilon^{\beta}}X^\varepsilon_{t+ \varepsilon^{2\beta}s}, \quad s\in (0, T^\varepsilon].
\end{equation*}
The dynamics of $\wt{X}^{\varepsilon, t}$ is given by 
\begin{equation*}\label{pb: local dynamics integral}
\wt{X}^{\varepsilon, t}_s = \wt{X}^{\varepsilon, t}_0 - \int_0^s \wt{b}^{\varepsilon, t}_\nu d\nu - \int_0^s\sqrt{\wt{a}^{\varepsilon, t}_\nu}d\wt{W}^{\varepsilon, t}_\nu + \int_0^s\wt{u}^{\varepsilon, t}_\nu d\nu+ \sum_{0< \wt{\tau}^{\varepsilon, t}_j \leq s}\wt{\xi}^{\varepsilon}_j,
\end{equation*}
with $(\wt{W}^{\varepsilon, t}_s)$  Brownian motion w.r.t. $\wt{\calF}^{\varepsilon, t}_s = \calF_{t + \varepsilon^{2\beta}s}$, 
$$
\wt{b}^{\varepsilon, t}_s = \varepsilon^{\beta} b_{t+ \varepsilon^{2\beta}s},\quad \wt{a}^{\varepsilon, t}_s = a_{t + \varepsilon^{2\beta}s},$$
and 
\begin{equation*} \wt{u}^{\varepsilon, t}_s = \varepsilon^{\beta} u^\varepsilon_{t+ \varepsilon^{2 \beta}s},  \quad \wt{\xi}^{\varepsilon}_j= \frac{1}{\varepsilon^{\beta}}\xi^\varepsilon_j, \quad \wt{\tau}^{\varepsilon, t}_j = \frac{1}{\varepsilon^{2 \beta}}(\tau^\varepsilon_j-t)\vee 0.
\end{equation*}
The corresponding local cost is defined by 
\begin{equation*}
{I}^\varepsilon_t =\frac{1}{T^\varepsilon}\Big(
\int_0^{T^\varepsilon}(r_{t} D(\wt{X}^{\varepsilon,t}_s) +l_t Q(\wt{u}^{\varepsilon, t}_s))ds+ \sum_{0<\wt{\tau}^{\varepsilon,t}_j\leq T^\varepsilon }(k_{t}F(\wt{\xi}^\varepsilon_j)+  h_{t}P(\wt{\xi}^\varepsilon_j))\Big).
\end{equation*}
Thanks to the following lemma, proven in \cite[Section
6]{cai2015asymptotic} , to characterize the asymptotic behavior of the cost functional, it is
enough to study the local cost.
\begin{lem} We have
\begin{equation*}
\lim_{\varepsilon\to 0}\frac{1}{\varepsilon^{\zeta_D \beta}}  J^\varepsilon  =
\lim_{\varepsilon\to 0}\int_0^TI^\varepsilon_t dt, \quad  a.s.,
\end{equation*}
if the term on the right hand side exists.
\end{lem}
\noindent  Moreover, up to a localization procedure, we can assume that $U_t^g=\{(x, U_t(x)), x\in G_t\}$ is contained in a bounded ball of $\bbR^d\times \bbR^d$, and there exist positive constants $\delta$ and $M$ such that
$
\delta< V(\xi_t(x))- V(x) $, $ x\in \partial G_t$,
$
V(x) < M$, $ x\in G_t$,
and moreover 
$$
\|a_t(\omega)\|\vee r_t(\omega)^{\pm 1}\vee l_t(\omega)^{\pm 1}\vee h_t(\omega)^{\pm 1}\vee k_t(\omega)^{\pm 1}<M
$$
for any $(t, \omega)$. Finally, since we are interested in convergence
in probability we may assume that $X^\circ$ is a martingale, that is,
$b_t\equiv 0$. \\

\noindent Define random measures $(\mu^\varepsilon_t, \rho^\varepsilon_t)$ by
\begin{align}
\mu^\varepsilon_t&=\frac{1}{T^\varepsilon}\int_0^{T^\varepsilon}\delta_{\{(\wt{X}^{\varepsilon, t}_s,\wt{u}^{\varepsilon,t}_s)\}}ds \in \scrP(\bbR^d\times \bbR^d), \nonumber \\
\rho^\varepsilon_t& = \frac{1}{T^\varepsilon} \sum_{0 < \wt{\tau}^{\varepsilon, t}_j \leq T^\varepsilon} \delta_{\{(\wt{X}^{\varepsilon, t}_{\wt{\tau}^{\varepsilon, t}_j-}, \wt{\xi}^{\varepsilon}_j)\}} \in \scrM(\bbR^d\times\bbR^d) ),\nonumber
\end{align}
and   $c: \Omega \times [0, T] \times \scrP(\bbR^d\times \bbR^d)\times \scrM({\bbR^d\times \bbR^d})\to \bbR$   by
\begin{align*}
c_t(\omega,\mu, {\rho}) :=& \int_{\bbR^d\times \bbR^d}(r_t(\omega) D(x)
+ l_t(\omega) Q(u)) \mu(dx\times du) \\&\qquad+ \int_{{\bbR^d\times \bbR^d}} (k_t(\omega) F(\xi) + h_t(\omega) P(\xi)){\rho}(dx\times d\xi).
\end{align*}
Then $(\mu^\varepsilon, \rho^\varepsilon)$ is a sequence of stochastic processes with values in $\scrP(\bbR^d\times \bbR^d)\times \scrM({\bbR^d\times \bbR^d})$ and
\begin{equation*}
\int_0^TI^\varepsilon_t dt= \int_0^T c_t(\mu^\varepsilon_t, \rho^\varepsilon_t) dt.
\end{equation*}
Let $(\pi_t, \nu_t):= (\pi^{(a_t, U_t, G_t, \xi_t)}, \nu^{(a_t, U_t, G_t, \xi_t)})$ be the process uniquely determined by \eqref{eqn: separability}, and put
\begin{equation*}
\mu_t(dx\times du) = \pi_t(dx)\otimes \delta_{U(x)} (du), \quad \rho_t(dx\times d\xi) = \nu_t(dx)\otimes \delta_{\xi(x)}(d\xi).
\end{equation*}
Then we have to show
\begin{equation}\label{eqn: global convergence}
\int_0^T c_t(\mu^\varepsilon_t, \rho^\varepsilon_t) dt \to_p \int_0^T c_t(\mu_t, \rho_t) dt.
\end{equation}
In view of Appendix \ref{sec: conv of integrals}, it suffices to prove the following lemma. 

\begin{lem}\label{lem: local to global}
\begin{enumerate}
\item
For any $t\in [0, T)$,  $(\mu^\varepsilon_t, \rho^\varepsilon_t)$ converges in probability to $(\mu_t, \rho_t)$. 
\item
The sequence $\{(\mu^\varepsilon_t, \rho^\varepsilon_t)_{t\in [0, T]}, \varepsilon>0\}$ of measure valued stochastic processes is weakly tight  with respect to  the cost functional $c$ (see Definition \ref{dfn: weakly tight}). 
\end{enumerate}
\end{lem}


\begin{proof}
We claim first that
\begin{equation}\label{eqn: uniform integrable}
\sup_{\varepsilon >0}\bbE\Big[\Big(\frac{N^\varepsilon_{T^\varepsilon}}{T^\varepsilon}\Big)^k\Big] < \infty,\quad k=1, 2, 
\end{equation}
Indeed, by Ito formula, we have
\begin{align*}
\frac{{N^\varepsilon_{T^\varepsilon}}}{T^\varepsilon}
&\leq \frac{1}{T^\varepsilon}{\sum_{j=1}^{N^\varepsilon_{T^\varepsilon}}\delta^{-1} (V(\wt{X}^{\varepsilon, t}_{\wt{\tau}^\varepsilon_j-}) -V(\wt{X}^{\varepsilon, t}_{\wt{\tau}^\varepsilon_j}) )}\\
&= \delta^{-1}\frac{1}{T^\varepsilon}\Big[-
V(\wt{X}^{\varepsilon, t}_{T^\varepsilon} )+V(\wt{X}^{\varepsilon, t}_{0+}) - \int_0^{T^\varepsilon} \nabla V(\wt{X}^{\varepsilon, t}_s)^T  \sqrt{\wt{a}^{\varepsilon, t}_s}d\wt{W}^{\varepsilon, t}_s \\
&\qquad    + \int_0^{T^\varepsilon} \sum_{i}\wt{u}^{\varepsilon, t}_{i, s}\partial_{i}V(\wt{X}^{\varepsilon, t}_s)ds + \int_0^{T^\varepsilon} \frac{1}{2}\sum_{ij}\wt{a}^{\varepsilon, t}_{ij, s}\partial^2_{ij}V(\wt{X}^{\varepsilon, t}_s)ds 
\Big] \\
&\leq \frac{M}{\delta}\Big(const. -  \int_0^{T^\varepsilon} \nabla V(\wt{X}^{\varepsilon, t}_s)^T  \sqrt{\wt{a}^{\varepsilon, t}_s}d\wt{W}^{\varepsilon, t}_s \Big)
\end{align*}
the last term being obviously square-integrable after the localization procedure.  Now  we are ready to prove the two claims. \\

\emph{Convergence in probability. } By \eqref{eqn: uniform integrable} and localization, for any $t\in [0, T)$ fixed,  the family $\{(\mu^\varepsilon_t, \rho^\varepsilon_t), \varepsilon >0\}$ is tight. 
Let $\bbQ_t$ be any stable limit of $(\mu^\varepsilon_t, \rho_t^\varepsilon)$. Since
\begin{equation*}
c_t(\mu^\varepsilon_t, \rho^\varepsilon_t) \leq M(1 + \frac{N^\varepsilon_{T^\varepsilon}}{T^\varepsilon}),
\end{equation*} 
we have $\sup_\varepsilon \bbE[c_t(\mu^\varepsilon_t,
\rho^\varepsilon_t)]< \infty$ in view of \eqref{eqn: uniform
  integrable} with $k=1$. Then by Lemma 6.3 in \cite{cai2015asymptotic}, we have $\bbQ_t^\omega$-a.e., $$(\mu, \rho) \in S(a_t(\omega)),$$
where we recall that $S(a)$ is defined by
\begin{align*}
S(a)= \Big\{ (\mu, \bar{\rho})&\in \scrP({\bbR^d\times \bbR^d}) \times \scrM(\overline{\bbR^d\times \bbR^d\setminus\{0_\xi\}}), \\
&\bar{\rho}  = \rho + \theta_{\bar{\rho}} \delta_{\infty} \text{ with } \rho\in \scrM(\bbR^d\times \bbR^d\setminus\{0_\xi\}),   \\
& \int_{\bbR^d\times \bbR^d} A^af (x, u)\mu(dx, du)  + \int_{\bbR^d\times \bbR^d\setminus\{0_\xi\}} Bf(x, \xi)\rho(dx, d\xi) = 0, \forall f\in C^2_0(\bbR^d)\Big.\Big\}.
\end{align*} 
On the other hand, let $F^\omega_\mu$ and $F^\omega_\rho$ be given by
\begin{equation*}
F^\omega_\mu =\{(x, U_t^\omega(x))|x\in G^\omega_t \}, \quad F^\omega_\rho=\{(x, \xi_t^\omega(x))|x\in \partial G^\omega_t \}.
\end{equation*}
By the continuity of $(G_t, \xi_t)$, $\bbQ^\omega_t$-a.e., $(\mu, \rho)$ is supported on $F^\omega_\mu$ and $F^\omega_\rho$ respectively. By the separability \eqref{eqn: separability} of $(G_t, \xi_t)$, such couple of $(\mu, \rho)$ is unique, so we have
\begin{equation*}
\bbQ_t = \bbP(d\omega)\otimes\delta_{\{(\mu^{(a_t,U_t, G_t, \xi_t)(\omega)}, \rho^{(a_t(, U_t, G_t, \xi_t)(\omega)})\}}.
\end{equation*}
which is the unique possible limit point and we deduce that convergence in probability holds. \\

\emph{Weak tightness of $(\mu^\varepsilon, \rho^\varepsilon)$ with respect to $c$. } 
In view of \eqref{eqn: uniform integrable} with $k=2$, the application $(t, \omega)\mapsto c_t(\omega, \mu^\varepsilon_t(\omega), \rho^\varepsilon_t(\omega))$ is uniformly square-integrable w.r.t. $\bbP\otimes dt$, hence uniformly integrable. By Lemma 2 of \cite{cremers1986weak} and the remark after it, $(\mu^\varepsilon, \rho^\varepsilon)$ is weakly tight w.r.t. $c$. 
\end{proof}

\paragraph{Proof of  Theorem \ref{theo: limit in proba singular}}
The proof is the same as Theorem \ref{theo: limit in proba} with \eqref{eqn: uniform integrable} replaced by
\begin{equation*}
\sup_{\varepsilon >0} \bbE\Big[\Big(\frac{\wt{\varphi}^\varepsilon_{T^\varepsilon}}{T^\varepsilon}\Big)^k\Big] < \infty, \quad k=1, 2,
\end{equation*}
which can be obtained by applying It\=o formula to $V(\wt{X}^\varepsilon_{T^\varepsilon})$ : 
\begin{align*}
\frac{{\wt{\varphi}^\varepsilon_{T^\varepsilon}}}{T^\varepsilon}
&\leq \frac{1}{T^\varepsilon}\int_0^{T^\varepsilon}\delta^{-1}[-\wt{\gamma}^\varepsilon_t \cdot \nabla V(\wt{X}^{\varepsilon, t}_s)]d\wt{\varphi}^\varepsilon_t\\
&= \delta^{-1}\frac{1}{T^\varepsilon}\Big[-
V(\wt{X}^{\varepsilon, t}_{T^\varepsilon} )+V(\wt{X}^{\varepsilon, t}_{0+}) + \int_0^{T^\varepsilon} \nabla V(\wt{X}^{\varepsilon, t}_s) \cdot  \sqrt{\wt{a}^{\varepsilon, t}_s}d\wt{W}^{\varepsilon, t}_s \\
&\qquad    + \int_0^{T^\varepsilon} \sum_{i}\wt{u}^{\varepsilon, t}_{i, s}\partial_{i}V(\wt{X}^{\varepsilon, t}_s)ds + \int_0^{T^\varepsilon} \frac{1}{2}\sum_{ij}\wt{a}^{\varepsilon, t}_{ij, s}\partial^2_{ij}V(\wt{X}^{\varepsilon, t}_s)ds 
\Big] \\
&\leq \frac{M}{\delta}\Big(const. +  \int_0^{T^\varepsilon} \nabla V(\wt{X}^{\varepsilon, t}_s)^T  \sqrt{\wt{a}^{\varepsilon, t}_s}d\wt{W}^{\varepsilon, t}_s \Big).
\end{align*}


\paragraph{Proof of Theorem \ref{theo: limit in proba stoch}}

The rescaled process $(\wt{X}^{\varepsilon, t}_s)$ is given by 
\begin{equation*}
d\wt{X}^{\varepsilon, t}_s = \wt{b}^{\varepsilon, t}_s ds  + \sqrt{\wt{a}^{\varepsilon, t}_s} d\wt{W}^{\varepsilon, t}_s + \wt{u}^{\varepsilon, t}_sds, 
\end{equation*}
with 
\begin{equation*}
\wt{b}^{\varepsilon, t}_s = \varepsilon^\beta_{t+ \varepsilon^{2\beta} s} , \quad  \wt{a}^{\varepsilon, t}_s = a_{t + \varepsilon^{2\beta s}}, \quad \wt{u}^{\varepsilon, t}_s =\varepsilon^{(2-1)\beta}u^\varepsilon_{t+\varepsilon^{2\beta} s}. 
\end{equation*}
The empirical occupation measure $\mu^\varepsilon_t$ is defined by
\begin{equation*}
\mu^\varepsilon_t = \frac{1}{T^\varepsilon}\int_0^{T^\varepsilon} \delta_{\{(\wt{X}^{\varepsilon, t}_s, \wt{u}^{\varepsilon, t}_s)\}}ds.
\end{equation*}
The proof is slightly different since $X^\varepsilon$ is not constrained inside a uniformly bounded domain $G_t$ as before. \\


\noindent Up to a localization procedure, we can assume that there exist $\theta, \Theta >0$ such that
\begin{equation*}
(A^{a_t}+ U_t(x)\cdot\nabla)V(x) \leq \theta - 2\Theta V(x), \quad x\in \bbR^d.
\end{equation*}
We follow \citep[Lemma 2.5.5]{arapostathis2011ergodic} and obtain that
\begin{equation}\label{eqn: bound V}
\E{V(\wt{X}^{\varepsilon,t}_s)} \leq \frac{\theta}{2\Theta} + V(0), \quad \forall t\in [0, T), s\in(0, T^\varepsilon].
\end{equation}
Let $\pi^\varepsilon_t$ be the marginal of $\mu^\varepsilon_t$ on $\bbR^d$. Since $ rD + l Q\circ U_t \leq V$ near infinity,  the empirical costs $\{c_t(\mu_t^\varepsilon)\}$ are bounded by $\int (b_0 + b_1 V(x)) d\pi^\varepsilon_t.$ By \eqref{eqn: bound V}, we have
\begin{equation}\label{eqn: bound c_t}
\sup_{\varepsilon}\bbE[c_t(\mu^\varepsilon_t)] \leq \sup_{\varepsilon} \bbE[\int (b_0 + b_1 V(x)) d\pi^\varepsilon_t] = b_0 + b_1\frac{1}{T^\varepsilon}\int_0^{T^\varepsilon}\bbE[V(\wt{X}^{\varepsilon, t}_s)]ds <\infty. 
\end{equation}

\noindent \emph{Convergence in probability}.  Since $U_t$ is admissible,  $U_t$ sends compact sets into compact sets, the tightness of $\mu^\varepsilon_t$ follows directly from the tightness of $\pi^\varepsilon_t$. 
By \eqref{eqn: bound c_t} we can apply Lemma 6.3 in \cite{cai2015asymptotic} and the convergence in probability follows from the separability of $\pi$ w.r.t. $A^a_u$. \\

\noindent \emph{Weak tightness.} For any $t\in [0, T)$ and $\varepsilon>0$, we have that \eqref{eqn: bound c_t} holds for a sufficiently large constant.  Since $c_t(\mu_t^\varepsilon)\to_p c_t(\mu_t)$, we have by Fubini theorem and dominated convergence theorem
\begin{equation*}
\limsup_{\varepsilon\to 0} \int c_t(\mu_t^\varepsilon) dt\otimes d\bbP = \int c_t(\mu_t)dt\otimes d\bbP < \infty.
\end{equation*}
By \cite[Lemma 3]{cremers1986weak}, we obtain the weak tightness.




\begin{subappendices}
\section{Convergence of integral functionals} \label{sec: conv of integrals}
In this section, we provide a direct generalization of the result in \cite{cremers1986weak}, which allows us to pass from the convergence of local systems (Lemma \ref{lem: local to global})  to the convergence of cost integrals \eqref{eqn: global convergence}. \\

\noindent Let $(\Omega, \calF, \bbP)$ be a probability space, $(T, \calB, \mu)$ a $\sigma$-finite measure space and $S$ a Polish space with Borel $\sigma$-field $\calB_S=\calB(S)$, $C(S) (C_b(S))$ the space of continuous (bounded continuous) real-valued functions on $S$, $\calL_1(\mu):=\calL_1(T, \calB, \mu)$ the space of integrable real-valued functions with seminorm $\|x\|_1:=\int_T |x_t|\mu(dt)$ and $L_1(\mu):=L_1(T, \calB, \mu)$ the corresponding Banach space. \\

\noindent Now let $c: \Omega\times T\times S\to \bbR$ be a $\calF\otimes\calB\otimes\calB_S$-measurable function with $c_t(\omega, \cdot)\in C(S)$ for all $t\in T$.  Let $\{X^n, n \in \bbN\}$ be a sequence of $\calF\otimes \calB/\calB_S$-measurable functions $X^n: \Omega\times T\to S$ with
\begin{equation*}
\int_T |c_t(\omega, X^n_t(\omega))|\mu(dt) < +\infty, \quad \omega\in \Omega, n\in \bbN.
\end{equation*}
Then the random integral $I(X^n, c)$ is defined by 
\begin{equation*}
I(X^n, c):= \int_T c_t(X^n_t)\mu(dt). 
\end{equation*}
Note that the map $(\omega, t)\mapsto X^{n, c}_t(\omega):= c_t(\omega, X^n_t(\omega))$ is $\calF\otimes \calB$-measurable. For all $\omega\in \Omega$, $\hat{X}^{n, c}:=X^{n, c}_\cdot(\omega)$ is an element of $L_1(\mu)$  and $\omega\mapsto \hat{X}^{n, c}(\omega)$ is $\calF/\calB(L_1(\mu))$-measurable (see \cite{cremers1986weak}).

\begin{dfn}
The processes $X^n$ converges to $X^0$ in probability in finite dimension if there is $T_0\in \calB$ with $\mu(T\setminus T_0)=0$ such that for any $t_1, \cdots, t_k\in T_0$, 
$(X^n_{t_1}, \cdots, X^n_{t_k})$ converges to  $(X^0_{t_1}, \cdots, X^0_{t_k})$ in probability. 
\end{dfn}

\begin{dfn}\label{dfn: weakly tight}
A sequence $\{X^n, n\in\bbN\}$ of measurable processes is called weakly tight with respect  to $c$ if for each $\delta >0$, there is $K\subset L_1(\mu)$ weakly compact, that is compact in the $\sigma(L_1(\mu), L_\infty(\mu))$-topology, such that
\begin{equation*}
\inf_n \bbP[\hat{X}^{n, c} \in K] > 1-\delta. 
\end{equation*}
\end{dfn}

\noindent  In particular, the family $X^n$ is weakly tight with respect to  $c$ if one of the following condition holds:
\begin{enumerate}
\item
The family of r.v. $\{(\omega, t)\mapsto c_t(\omega, X^n_t(\omega)),  n  \in \bbN^*\}$ is $\bbP\otimes \mu$-uniformly integrable, see \cite[Lemma 2]{cremers1986weak}.
\item
 $X^n$  converges weakly to $X^0$ in finite dimension and
\begin{equation*}
\limsup_n \int c_t(X^n_t) dt\otimes d\bbP \leq \int c_t(X^0_t) dt\otimes d\bbP < \infty. 
\end{equation*}
See \cite[Lemma 3]{cremers1986weak}.
\end{enumerate}
Now we can state our result concerning the convergence in probability of the random variables $I(X^n, c)$. 

\begin{theo}
Let $\{X^n, n\in \bbN\}$ be a sequence of stochastic process. If $X^n$ converges to $X^0$ in probability in finite dimension and if $\{X^n, n\in \bbN\}$ is weakly tight w.r.t. $c$, then $I(X^n, c)$ converges to $I(X^0, c)$ in  probability. 
\end{theo}

\begin{proof}
After  Lemma D.1 in \cite{cai2015asymptotic}, it suffices to show that $I(X^n, c)$ converges stably to $I(X^0, c)$. Let $Y$ be any bounded r.v. and $f$ bounded continuous and Lipschitz. We will show that
\begin{equation}\label{eqn: conv temp}
\bbE[Yf(I(X^n, c))] \to \bbE[Y f(I(X^0, c))], \quad n\to \infty. 
\end{equation}

\noindent Step i.) Let $h\in \mathcal{L}_1^+(\mu)$. Define
\begin{equation*}
c^h_t(\omega, x) := \max\{ -h_t, \min  \{h_t, c_t(\omega, x)\}\},
\end{equation*}
and 
\begin{equation*}
I(X^n, c^h) := \int_T c^h_t(X^n_t)\mu(dt), \quad  n\in \bbN.
\end{equation*}
 We show that $I(X^n, c^h)$ converges stably to $I(X^0, c^h)$, i.e. 
\begin{equation}\label{eqn: conv h}
\bbE[Y f(I(X^n, c^h)] \to \bbE[Y f(I(X^0, c^h)],\quad n \to \infty.
\end{equation}
 Since $|I(X^n, c^h)|\leq \|h\|_1$ and polynomials are dense in $C([-\|h\|_1, \|h\|_1])$, we only need to consider $f(x) = x^l$ for some $l\in \bbN$. By Fubini's theorem, we obtain
\begin{align*}
\bbE[Y f(I(X^n, c^h))]
&= \int_\Omega Y(\omega) \Big(\int_Tc^h_{t_1}(\omega, X^n_{t_1}) \mu(d{t_1})\cdots \int_T Y(\omega) c^h_{t_l}(\omega, X^n_{t_l}) \mu(d{t_l})\Big) \bbP(d\omega)\\
&=\int\cdots\int F_n(t_1, \cdots, t_l) \mu(dt_1)\cdots \mu(dt_l)
\end{align*}
where
\begin{equation*}
F_n(t_1, \cdots, t_l) := \int_\Omega Y(\omega) c^h_{t_1}(\omega, X^n_{t_1})\cdots c^h_{t_l}(\omega, X^n_{t_l})\bbP(d\omega).
\end{equation*}
By weak convergence and the finite dimensional convergence in probability of $X^n$, we have $F_n(t_1, \cdots, t_l)\to F_0(t_1, \cdots, t_l)$ for any $t_1, \cdots, t_l \in T_0$. Since $$F_n(t_1, \cdots, t_l) \leq (\sup_\omega |Y|)h(t_1)\cdots h(t_l),$$ we have
\begin{equation*}
\int\cdots\int F_n(t_1, \cdots, t_l) \mu(dt_1)\cdots \mu(dt_l)\to \int\cdots\int F_0(t_1, \cdots, t_l) \mu(dt_1)\cdots \mu(dt_l),
\end{equation*}
by dominated convergence, whence \eqref{eqn: conv temp}.\\

\noindent Step ii.) By \cite[Remark 1]{cremers1986weak} and the weak tightness of
$X^n$ w.r.t. $c$, for all $N\in \bbN^*$ there is a weakly compact set $K_N$ of  $L_1(\mu)$ and $h_N\in \calL^+_1(\mu)$  such that
\begin{equation}\label{eqn: temp 1}
\inf_{n} \bbP[\hat{X}^{n, c} \in K_N] \geq 1- \frac{1}{N},
\end{equation}
and
\begin{equation}\label{eqn: temp 2}
\sup_{x\in K_N} \int_T (\abs{x} - h_N)^+ d\mu \leq \frac{1}{N}.
\end{equation}
We can assume w.l.o.g that $h_N\uparrow \infty$ for $N\uparrow \infty$. 
We have
\begin{align*}
 {|\bbE[Yf(I(X^n, c))] - \bbE[Yf(I(X^0, c))]|}
&\leq  {|\bbE[Yf(I(X^n, c))] - \bbE[Yf(I(X^n, c^{h_N}))]|} \\
&\quad+ {|\bbE[Yf(I(X^n, c^{h_N}))] - \bbE[Yf(I(X^0, c^{h_N}))]|} \\
&\quad + {|\bbE[Yf(I(X^0, c^{h_N}))] - \bbE[Yf(I(X^0, c))]|}\\
&=: e_1 +e_2 + e_3.
\end{align*}
Since $Y$ is bounded and $f$ is bounded Lipschitz, we have
\begin{align*}
 e_1 &\leq (\sup_\omega |Y|)  (2(\sup_x |f|) \bbP[|I(X^n, c) - I(X^n, c^{h_N})| \geq \frac{1}{N}] + \text{Lip}(f) \frac{1}{N} )\\
& \leq  const. \frac{1}{N},
\end{align*}
where the constant is independent of $n$. Here we use \eqref{eqn: temp 1}, \eqref{eqn: temp 2} and 
\begin{align*}
|I(X^n, c^{h_N}) - I(X^n, c)|
& \leq \int_T |c^{h_N}_t(X^n_t) - c_t(X^n_t)| \mu(dt) \\
&= \int_T(|c_t(X^n_t)| - h_{N}(t))^+ \mu(dt).
\end{align*}
Hence $e_1<\varepsilon$ for $N$ large enough. By dominated convergence, $I(X^0, c^{h_N})\to I(X^0, c)$ pointwise, hence $e_3< \varepsilon$ for $N$ large enough. Now fix $N$, by \eqref{eqn: conv h}, we have $e_2 < \varepsilon$ for $n $ large enough.
In sum, we have $|\bbE[Yf(I(X^n, c))] - \bbE[Yf(I(X^0, c))]| \leq 3 \varepsilon$ for $n$ large enough. Since $\varepsilon$ is arbitrary, \eqref{eqn: conv temp} follows and we can conclude.
\end{proof}

\end{subappendices}



\begin{thebibliography}{10}

\bibitem{arapostathis2011ergodic}
{\sc A.~Arapostathis, V.~S. Borkar, and M.~K. Ghosh}, {\em Ergodic control of
  diffusion processes}, no.~143, Cambridge University Press, 2011.

\bibitem{cai2015asymptotic}
{\sc J.~Cai, M.~Rosenbaum, and P.~Tankov}, {\em Asymptotic lower bounds for
  optimal tracking: a linear programming approach}.
\newblock Arxiv preprint, 2015.

\bibitem{cremers1986weak}
{\sc H.~Cremers and D.~Kadelka}, {\em On weak convergence of integral
  functionals of stochastic processes with applications to processes taking
  paths in \{LEP\}}, Stochastic processes and their applications, 21 (1986),
  pp.~305--317.

\bibitem{freidlin85functional}
{\sc M.~I. Freidlin}, {\em Functional Integration and Partial Differential
  Equations}, Annals of Mathematics Studies 109, Princeton University Press,
  1985.

\bibitem{gobet2012almost}
{\sc E.~Gobet and N.~Landon}, {\em Almost sure optimal hedging strategy}, The
  Annals of Applied Probability, 24 (2014), pp.~1652--1690.

\bibitem{kisielewicz2013stochastic}
{\sc M.~Kisielewicz}, {\em Stochastic Differential Inclusions and
  Applications}, Springer, 2013.

\bibitem{kurtz2001stationary}
{\sc T.~G. Kurtz and R.~H. Stockbridge}, {\em Stationary solutions and forward
  equations for controlled and singular martingale problems}, Electronic
  Journal of Probability, 6 (2001), pp.~1--52.

\bibitem{ladyzhenskaya1968linear}
{\sc O.~A. Ladyzhenskaya and N.~N. Ural'tseva}, {\em Linear and quasilinear
  elliptic equations}, Academic Press, 1968.

\bibitem{nystrom2010skorohod}
{\sc K.~Nystr\"{o}m and T.~\"{O}nskog}, {\em The {S}korohod oblique reflection
  problem in time-dependent domains}, The Annals of Probability, 38 (2010),
  pp.~2170--2223.

\bibitem{protter2004stochastic}
{\sc P.~E. Protter}, {\em Stochastic Integration and Differential Equations},
  Springer, 2nd~ed., 2004.

\bibitem{slominski2013stochastic}
{\sc L.~S{\l}omi\'nski and T.~Wojciechowski}, {\em Stochastic differential
  equations with time-dependent reflecting barriers}, Stochastics An
  International Journal of Probability and Stochastic Processes, 85 (2013),
  pp.~27--47.

\bibitem{stroock1971diffusion}
{\sc D.~W. Stroock and S.~Varadhan}, {\em Diffusion processes with boundary
  conditions}, Communications on Pure and Applied Mathematics, 24 (1971),
  pp.~147--225.

\bibitem{tweedie1975sufficient}
{\sc R.~L. Tweedie}, {\em Sufficient conditions for ergodicity and recurrence
  of {M}arkov chains on a general state space}, Stochastic Processes and their
  Applications, 3 (1975), pp.~385--403.

\end{thebibliography}

\end{document}